\documentclass{amsart}
\usepackage{tikz,bm,amsmath,amssymb,amsrefs}
\newtheorem{theorem}{Theorem}
\newtheorem{lemma}{Lemma}
\newtheorem{example}{Example}
\newtheorem{definition}{Definition}
\newtheorem{cor}{Corollary}

\newtheorem{proposition}{Proposition}
\newtheorem*{claim}{Claim}

\def\ld{\mathbin{\backslash}}
\def\rd{\mathbin{/}}

\newcommand{\bl}{\underline{\hspace{1ex}}}

\newcommand{\blambda}{\boldsymbol{\lambda}}
\newcommand{\brho}{\boldsymbol{\rho}}

\title{Beyond wreath and block}

\author[Botur]{Michal Botur$^\dag$}
\author[Kowalski]{Tomasz Kowalski$^\ddag$}

\address{$^\dag$ Chair of Algebra and Geometry, Palacky University, Olomouc}
\email{michal.botur@upol.cz}

\address{$^\ddag$ Department of Logic, Jagiellonian University, Kraków}
\email{tomasz.s.kowalski@uj.edu.pl}


\begin{document}

\begin{abstract}
We investigate a semigroup construction generalising the two-sided
wreath product.  We develop the foundations of this construction
and show that for groups it is isomorphic to the usual
wreath product. We also show that it gives a 
slightly finer version of the decomposition in the 
Krohn-Rhodes Theorem, in which the three-element flip-flop monoid is replaced by
the two-element semilattice. 
\end{abstract}  

\maketitle

\section{Introduction}\label{intro}

The purpose of this article is to introduce and investigate a certain semigroup
construction which encompasses a range of known constructions including
wreath products and block products.
The construction, which we call a $\lambda\rho$-\emph{product},
is inspired by the standard way of presenting the wreath product, 
say, of groups, as a direct power $G^X$ together with a group $K$ acting on $X$,
that is, a set of bijective maps $X\to X$, indexed by
elements of $K$. For semigroups, the restriction to bijections seems artificial: after
all, semigroups are representable as semigroups of arbitrary maps. And if the
maps do not have to be surjective, there seems to be no reason for having the
same set of coordinates for every element of $K$.  

A rudimentary construction of this type has been used in~\cite{JM06}
to settle some questions about \emph{generalised BL-algebras}, which are
a subclass of certain special lattice-ordered monoids known as
\emph{residuated lattices}. For the purposes of this article, familiarity with
residuated lattices is not necessary, but the interested reader is referred
to~\cite{JT02} for a very readable albeit slightly old survey.

The construction was expanded and
investigated in~\cite{DK14}, under the name of \emph{kites}, still in the
context of residuated lattices. A very simple example of a kite can be
informally described as follows. Start with $(\mathbb{Z}; \leq, +, 0)$ as a
lattice-ordered 
group. Take $\mathbb{Z}\times \mathbb{Z}$ and another copy of $\mathbb{Z}$;
extend the natural order on $\mathbb{Z}\times \mathbb{Z}$ and
$\mathbb{Z}$
by putting $\mathbb{Z}\times \mathbb{Z}$ on top of
$\mathbb{Z}$; truncate to the interval $[0, \langle 0,0\rangle]$. Products
in the top part are as in $\mathbb{Z}\times \mathbb{Z}$. The other products
are defined by
\begin{align*}
  \langle x,y\rangle\cdot i &= \max\{x+i,0\}\\
  i\cdot\langle x,y\rangle &=  \max\{y+i,0\}\\
  i\cdot j &= 0
\end{align*}
This behaviour can be described with the help of a two element semigroup
$\{a,b\}$ satisfying $a^2 = a$ and $uv = b$ for all other products. 
We think of $a$ as indexing the top part, $b$ as indexing the bottom part. 
Then, to give an alternative definition of product in our kite
we may use a set of maps $\lambda$ (the left maps) and
$\rho$ (the right maps) between the sets of coordinates, telling us which
coordinate to take for which product. Thus, 
$\langle x,y\rangle\cdot i$ can be presented as
$$
(\langle x,y\rangle, a)\cdot (i, b) =
\bigl((\langle x,y\rangle\circ \lambda[a,b])\cdot (i\circ \rho[a,b]), ab\bigr) 
$$
where $\lambda[a,b]\colon I[ab] \to I[a]$ is given by
$\lambda[a,b](0) = 0$, and $\rho[a,b]\colon I[ab] \to I[b]$ is given by
$\rho[a,b](0) = 0$ (as $ab = b$, we have $I[ab] = I[b]$). Calculating the
product will then give 
$$
\bigl((\langle x,y\rangle\circ \lambda[a,b])\cdot (i\circ \rho[a,b]), ab\bigr)
= (x + i, ab)
$$
which is precisely what we want, disregarding the truncation. Similarly
\begin{align*}
(i, b)\cdot (\langle x,y\rangle, a)
&= \bigl((i\circ \lambda[b,a])\cdot (\langle x,y\rangle\circ \rho[b,a]), ba\bigr)\\ 
&= (i + y, ba) 
\end{align*}  
where $\rho[b,a]\colon I[ba]\to I[a]$ is given by $\rho[b,a](0) = 1$,
and $\lambda[b,a]\colon I[ba]\to I[b]$ is the identity, of course.

The product defined this way is associative, but it also turns out to be
residuated, which makes the 
algebra just defined a residuated lattice. Another version of the same
construction arises by replacing the bottom copy of $\mathbb{Z}$, by
$\mathbb{Z}\times\mathbb{Z}$, truncating to the interval
$[\langle 0,0\rangle, \langle 0,0\rangle]$, and defining the products 
between the top and the bottom parts by
\begin{align*}
  \langle a,b\rangle\cdot \langle i,j\rangle  &=
      \max\{\langle a+j,b+i\rangle,\langle 0,0\rangle\}\\
  \langle i,j\rangle\cdot\langle a,b\rangle &=
      \max\{\langle a+i,b+j\rangle,\langle 0,0\rangle\}
\end{align*}
where $\langle a,b\rangle$ comes from the top and 
$\langle i,j\rangle$ comes from the bottom. Note the coordinate swap in one, but
not in the other. The products can also be defined via an appropriate system of maps
$\lambda$ and $\rho$ analogously to the previous example. The algebra obtained this way is
isomorphic (see~\cite{DK14}, Example~4.5) to a truncation of a subgroup $G$ of the
antilexicographically ordered wreath product $\mathbb{Z}\wr\mathbb{Z}$
consisting of the elements  $\langle\langle a_\ell : i\in\mathbb{Z}\rangle,
b\rangle$ such that $\ell = k \pmod 2$ implies $a_\ell = a_k$. 

The kite construction comes in two parts: the twisting and the truncation. The
truncation is important for certain order theoretic purposes, but it does not
play any role for the product definition, if the twisting is handled with care.

A series of applications and further
generalisations of the kite construction followed, see, e.g.,~\cite{DH14}
and~\cite{BD15}. 
Another modification was put to a good use in~\cite{BKLT16}. All these, however,
stayed within the area of ordered structures, and the interaction of
multiplication with order was the main focus. It was clear from the beginning
that the kite construction is closely related to wreath products of 
ordered structures, for example from~\cite{JT04}, or specifically for
lattice-ordered groups from~\cite{HMC69}. 
Considering order, however, seems
to have obscured the properties of the multiplicative structure to some extent.

Here we will not consider order at all and investigate only the multiplicative
structure. It will turn out that certain semigroups not decomposable by
standard constructions, are decomposable by ours. We will apply this 
to the celebrated Krohn-Rhodes Theorem (originally in~\cite{KR62}, see
also~\cite{RS09}), replacing the three-element monoid $L_2^1$ (or
$R_2^1$) by the two-element semilattice. We quickly admit that this application
piggybacks on Krohn-Rhodes Theorem: it uses its full force, adding only that
$L_2^1$ can be further decomposed using our construction.

We will also show that our construction applied to groups coincides with the
usual wreath product. This, we believe, shows the naturalness of the
construction, and we consider it the main result of the article. 

Section~\ref{constr} below defines $\lambda\rho$-systems and
$\lambda\rho$-products.  Section~\ref{cats} uses Grothendieck construction
to define a category of $\lambda\rho$-systems in a natural way.
Section~\ref{free-constr} gives a construction of a $\lambda\rho$-system
out of any family of sets, over a free semigroup generated by the index set
of the family. 
Section~\ref{simpli} considers $\lambda\rho$-systems over monoids, and
Section~\ref{wreath} shows that
$\lambda\rho$-products for groups coincide with wreath products. 

\section{The construction}\label{constr}

\subsection{Notation}
We use the category-theoretic notation for composition of maps, that is, 
for maps $f\colon A\longrightarrow B$ and $g\colon B\longrightarrow C$
we denote their composition by $g\circ f\colon A\longrightarrow C$, so that
$(g\circ f)(a) = g(f(a))$ for all $a\in A$. The set of all
maps from $A$ to $B$ we denote by the usual $B^A$. For a map
$f\colon A\longrightarrow B$ and a set $I$ we write
$f^I\colon A^I\to B^I$ for the map defined by $f^I(x)(i)=f(x(i))$.
The following easy proposition (in which by groupoid we mean
an algebra with a single binary operation) will be used repeatedly without further ado.

\begin{proposition}\label{P1}
Let $\mathbf{G} = (G;\cdot)$ be a groupoid, and let $I$, $J$ be sets. 
Then for all $x,y\in G^I$ and any $f\in I^J$ the following equality holds
$$(x\circ f)\cdot (y\circ f)= (x\cdot y)\circ f.$$
\end{proposition}

We will frequently use parameterised systems of maps. In order to
distinguish easily between parameters and arguments, we will put the parameters
in square brackets, so $f[a,b](x)$ will denote the value of a
map $f[a,b]$ on the argument $x$.

We will also frequently pass between algebras (semigroups), categories, and
other types of structures (systems of maps), typically related to one
another. To help distinguishing between them, we will use different fonts.
Typically, boldface will be used for algebras (and italics for their
universes), sans serif will be used for categories, and script for other types
of structures. A few exceptions to these rules will be natural enough not to
cause confusion.

\subsection{$\lambda\rho$-systems and $\lambda\rho$-products}

Let $\mathbf{S}$ be a semigroup. We will write $S$ for the universe of
$\mathbf{S}$, and use this convention systematically from now on. Let 
$(I[s])_{s\in S}$ be a system of sets indexed by the elements of $S$.
For any $(a,b)\in S^2$, let 
$\lambda[a,b]\colon I[ab]\to I[a]$ and $\rho[a,b]\colon I[ab]\to I[b]$
be maps satisfying the following conditions
\begin{enumerate}
\item[($\alpha$)] $\lambda[a,b]\circ\lambda[ab,c] = \lambda[a,bc]$
\item[($\beta$)] $\rho[b,c]\circ\rho[a,bc] = \rho[ab,c]$
\item[($\gamma$)] $\rho[a,b]\circ\lambda[ab,c] = \lambda[b,c]\circ\rho[a,bc]$
\end{enumerate}
which make the diagram in Figure~\ref{l-r-system} commute. 
\begin{figure}
\begin{tikzpicture}[>=stealth,auto]
\node (I_abc) at (0,0) {$I[abc]$};
\node (I_a) at (-4,0) {$I[a]$};
\node (I_c) at (4,0) {$I[c]$};
\node (I_ab) at (-2,-2) {$I[ab]$};
\node (I_bc) at (2,-2) {$I[bc]$};
\node (I_b) at (0,-4) {$I[b]$};
\draw[->] (I_abc) to node[swap] {$\lambda[a,bc]$} (I_a);
\draw[->] (I_abc) to node {$\rho[ab,c]$} (I_c);
\draw[->] (I_abc) to node[swap] {$\lambda[ab,c]$} (I_ab);
\draw[->] (I_abc) to node {$\rho[a,bc]$} (I_bc);
\draw[->] (I_ab) to node {$\lambda[a,b]$} (I_a);
\draw[->] (I_ab) to node[swap] {$\rho[a,b]$} (I_b);
\draw[->] (I_bc) to node[swap] {$\rho[b,c]$} (I_c);
\draw[->] (I_bc) to node {$\lambda[b,c]$} (I_b);
\end{tikzpicture}
\caption{A $\lambda\rho$-system}
\label{l-r-system}
\end{figure}
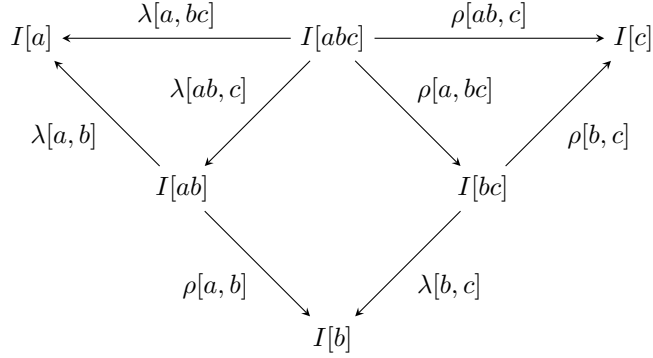
Let $\mathbf{S}$ be a semigroup, and let $\mathbf{I} =  (I[s])_{s\in S}$,
$\blambda = (\lambda[a,b]\colon I[ab]\to I[a])_{(a,b)\in S\times S}$ and
$\brho = (\rho[a,b]\colon I[ab]\to I[a])_{(a,b)\in S\times S}$
be a system of sets and maps satisfying the conditions above. 
We will call the triple $(\mathbf{I},\blambda,\brho)$ 
a \emph{$\lambda\rho$-system over} $\mathbf{S}$. A \emph{general
$\lambda\rho$-system} is then a pair $(\mathbf{S},\mathcal{S})$,
where $\mathbf{S}$ is a semigroup and $\mathcal{S}$
is a $\lambda\rho$-system over $\mathbf{S}$. 
We will typically use script letters to refer to $\lambda\rho$-systems,
together with the convention that a $\lambda\rho$-system over a semigroup will
be referred to by the script variant of the letter naming the semigroup.
Thus, a $\lambda\rho$-system over $\mathbf{S}$ will be generally called
$\mathcal{S}$; subscripts, and occasionally other devices, will be used to
distinguish between different  
$\lambda\rho$-systems over the same semigroup.
Where convenient, we will also use the more explicit notation
$$
\bigl(\langle \lambda[a,b],\rho[a,b]\rangle\colon
I[ab]\longrightarrow I[a]\times I[b]\bigr)_{(a,b)\in S^2}
$$
for a $\lambda\rho$-system over a semigroup $\mathbf{S}$.

If $\mathbf{S}$ is the trivial semigroup, then 
any $\lambda\rho$-system over $\mathbf{S}$ is a 
pair of commuting retractions on some set. 
Such $\lambda\rho$-systems were studied in~\cite{Bot21}
under the name of $\lambda\rho$-algebras, giving representations of certain
semigroups.

\begin{definition}\label{rl-prod}
Let $\mathbf{S}$ be a semigroup and let
$$
\mathcal{S} =   
\bigl(\langle\lambda[a,b],\rho[a,b]\rangle\colon 
I[ab]\to I[a]\times I[b]\bigr)_{(a,b)\in S^2}
$$
be a system of sets and maps indexed by the elements of $S^2$.
Let $\mathbf{H}$ be a semigroup. 
We define a groupoid  
$\mathbf{H}^{[\mathcal{S}]} = (H ^{[\mathcal{S}]};\star)$, by putting 
\begin{itemize}
\item $H^{[\mathcal{S}]} = \biguplus_{a\in S} H^{I[a]} = \{(x,a)\colon a\in S,\
  x\in H^{I[a]}\}$, and 
\item $(x,a)\star(y,b) = \bigl((x\circ\lambda[a,b])\cdot(y\circ\rho[a,b]),ab\bigr)$.
\end{itemize}
We call $\mathbf{H}^{[\mathcal{S}]}$ a \emph{$\lambda\rho$-product}.
\end{definition}

\begin{example}\label{skel}
Let $\mathbf{S}$ be a semigroup, and let $\mathbf{1}$ be the trivial semigroup.
Then, for any system $\mathcal{S}$ of sets and maps over $\mathbf{S}$ we have
$\mathbf{1}^{[\mathcal{S}]}\cong \mathbf{S}$. Indeed, $\mathbf{1}^I\cong
\mathbf{1}$ for any $I$, so
$\mathbf{1}^{[\mathcal{S}]} = \bigl(\{(1,s)\colon s\in S\},\star\bigr)$, with
$(1,a)\star(1,b) = (1,ab)$.  
\end{example}

The operation $\star$ in $\mathbf{1}^{[\mathcal{S}]}$ is associative only because  
$\mathbf{1}$ is the trivial semigroup. For an arbitrary semigroup 
$\mathbf{H}$, associativity of $\star$ in $\mathbf{1}^{[\mathcal{S}]}$
is equivalent to $\mathcal{S}$ being a $\lambda\rho$-system as we will now show.  

\begin{theorem}\label{main}
Let $\mathbf{S}$ be a semigroup and let
$$
\mathcal{S} = 
\bigl(\langle\lambda[a,b],\rho[a,b]\rangle\colon 
I[ab]\to I[a]\times I[b]\bigr)_{(a,b)\in S^2}
$$
be a system of sets and maps indexed by the elements of $S^2$. 
Then, the following are equivalent.
\begin{enumerate}
\item $\mathbf{H}^{[\mathcal{S}]}$ is a semigroup, for any semigroup $\mathbf{H}$.
\item $\bigl(\langle\lambda[a,b],\rho[a,b]\rangle\colon 
I[ab]\to I[a]\times I[b]\bigr)_{(a,b)\in S^2}$ is a $\lambda\rho$-system over $\mathbf{S}$.
\end{enumerate}
\end{theorem}

\begin{proof} First, note that associativity of 
  $\star$ is equivalent to the statement that the equality
\begin{equation}\label{assoc-eq}
\begin{split}
&\bigl((x\circ\lambda[a,b]\circ\lambda[ab,c])
   \cdot(y\circ\rho[a,b]\circ\lambda[ab,c]) 
   \cdot(z\circ\rho[ab,c]),\ abc\bigr) = \\
&\bigl((x\circ\lambda[a,bc])
   \cdot(y\circ\lambda[b,c]\circ\rho[a,bc]) 
   \cdot(z\circ\rho[b,c]\circ\rho[a,bc]),\ abc\bigr)
\end{split}
\end{equation} 
holds for arbitrary $(x,a), (y,b), (z,c)\in H^{[\mathcal{S}]}$. 
To see it, we carry out the following straightforward calculation:
\begin{align*}
\bigl((x,a)\star(y,b)\bigr)&\star(z,c) = 
\bigl((x\circ\lambda[a,b])\cdot(y\circ\rho[a,b]),\ ab\bigr)\star (z,c) \\
&= \biggl(\Bigl(\bigl((x\circ\lambda[a,b])\cdot(y\circ\rho[a,b])\bigr)
   \circ\lambda[ab,c]\Bigr) 
   \cdot\bigl(z\circ\rho[ab,c]\bigr),\ abc\biggr) \\
&= \bigl((x\circ\lambda[a,b]\circ\lambda[ab,c])
   \cdot(y\circ\rho[a,b]\circ\lambda[ab,c]) 
   \cdot(z\circ\rho[ab,c]),\ abc\bigr) \\
&= \bigl((x\circ\lambda[a,bc])
   \cdot(y\circ\lambda[b,c]\circ\rho[a,bc]) 
   \cdot(z\circ\rho[b,c]\circ\rho[a,bc]),\ abc\bigr)\\
&= \biggl(\bigl(x\circ\lambda[a,bc]\bigr)
   \cdot\Bigl(\bigl((y\circ\lambda[b,c]) 
   \cdot(z\circ\rho[b,c])\bigr)\circ\rho[a,bc]\Bigr),\ abc\biggr) \\
&= (x,a)\star
   \bigl((y\circ\lambda[b,c])\cdot(z\circ\rho[b,c]),\ bc\bigr)\\
&= (x,a)\star\bigl((y,b)\star(z,c)\bigr)  
\end{align*}
where the only non-definitional equality is precisely~(\ref{assoc-eq}).
Now, if $\mathcal{S}$ is a  $\lambda\rho$-system, then~(\ref{assoc-eq}) 
follows immediately from the equations ($\alpha$), ($\beta$) and ($\gamma$).
This proves that (2) implies (1).

For the converse, let $\mathbf{S}$ be a semigroup and let
$\mathcal{S}$ be a system of sets and maps over $\mathbf{S}$.
Let $H = \biguplus (I[a])_{a\in S}$ 
and take the free monoid $H^*$. Let
$id_a\colon I[a] \to H^*$ be the identity map on $I[a]$,
and let $\varepsilon_a\colon I[a] \to H^*$ be the constant map mapping
every element of $I[a]$ to the empty string. Then, we have
$(id_a,a), (\varepsilon_b,b), (\varepsilon_c,c)\in 
(H^*)^{[\mathcal{S}]}$. Calculating products in $(H^*)^{[\mathcal{S}]}$ gives:
\begin{align*}
\bigl((id_a,a)\star (\varepsilon_b,b)\bigr)\star (\varepsilon_c,c)
 &=  \bigl((id_a\circ \lambda[a,b])\cdot(\varepsilon_b\circ\rho[a,b]), ab\bigr)\star (\varepsilon_c,c)\\
  &=  \bigl((id_a\circ \lambda[a,b]), ab\bigr)\star (\varepsilon_c,c)\\
 &=  (id_a\circ \lambda[a,b]\circ\lambda[ab,c], abc)
\end{align*}  
and 
\begin{align*}
(id_a,a)\star \bigl((\varepsilon_b,b)\star (\varepsilon_c,c)\bigr)
 &=  (id_a,a)\star\bigl((\varepsilon_b\circ\lambda[b,c])\cdot(\varepsilon_c\circ\rho[b,c]), bc\bigr)\\
 &=  (id_a\circ \lambda[a,bc],abc)
\end{align*}
The left-hand sides are identical by the assumption that
$(H^*)^{[\mathcal{S}]}$ is a semigroup, so equating the 
right-hand sides we obtain
$$ 
id_a\circ \lambda[a,b]\circ\lambda[ab,c] = id_a\circ \lambda[a,bc]
$$
but as $id_a$ is injective, it can  be cancelled, giving
$$ 
\lambda[a,b]\circ\lambda[ab,c] = \lambda[a,bc]
$$
which shows that ($\alpha$) holds. Proofs of ($\beta$) and ($\gamma$) are
analogous. For ($\beta$) calculate
$(\varepsilon_a,a)\star (\varepsilon_b,b)\star (id_c,c)$ in two ways; for ($\gamma$)
calculate $(\varepsilon_a,a)\star (id_b,b)\star (\varepsilon_c,c)$ in two ways. 
\end{proof}

Note that in general neither
$\mathbf{S}$ nor $\mathbf{H}$ is a subsemigroup of 
$\mathbf{H}^{[\mathcal{S}]}$. However, it is not difficult to show that
if $\mathbf{H}$ has an idempotent, then $\mathbf{S}$ is a subsemigroup of
$\mathbf{H}^{[\mathcal{S}]}$, and if $\mathbf{S}$ has an idempotent $e$ such
that $I[e]\neq\emptyset$, then $\mathbf{H}$ is a subsemigroup of 
$\mathbf{H}^{[\mathcal{S}]}$.

\begin{example}\label{empty}
Let $\mathbf{S}$ be a semigroup, and let $I[s] = \emptyset$ for each $s\in S$.
Then, $\mathcal{S} = (\mathbf{I},\blambda,\brho)$, where $\lambda[a,b]$,
$\rho[a,b]$ are empty functions for each $(a,b)\in S^2$, is a
$\lambda\rho$-system over $\mathbf{S}$.  
For any semigroup $\mathbf{H}$ we then have that $H^{I[s]}$ is a singleton for each
$s\in S$ \textup{(}its only element is the empty map
$\emptyset\colon \emptyset\to H$\textup{)}. 
Moreover, $(\emptyset,a)\star(\emptyset,b) = (\emptyset, ab)$, for
any $a,b\in S$, and thus $\mathbf{H}^{[\mathcal{S}]}\cong \mathbf{S}$. 
\end{example}  

One can ask how much freedom there is for making some but not necessarily all
sets $I[s]$ empty. The answer, whose easy proof we leave to the reader,
is below.

\begin{proposition}\label{ideal}
Let $\mathcal{S} = (\mathbf{I},\blambda,\brho)$ be a $\lambda\rho$-system over a
semigroup $\mathbf{S}$. Let $J = \{s\in S\colon I[s] = \emptyset\}$.
If $J$ is nonempty, then $J$ is a two-sided ideal of\/ $\mathbf{S}$.
\end{proposition}  

\begin{example}\label{prod}
Let $\mathbf{S}$ be a semigroup, and let $I[s] = \{1\}$ for each $s\in S$.
Then, $\mathcal{S} = (\mathbf{I},\blambda,\brho)$, where $\lambda[a,b]$,
$\rho[a,b]$ are constant functions for each $(a,b)\in S^2$, is a
$\lambda\rho$-system over $\mathbf{S}$.  
Then, $H^{I[s]}$ is a copy of $H$, for any semigroup $\mathbf{H}$. 
Moreover, for 
any $a,b\in S$ and $x,y\in H$, we have 
$(x,a)\star(y,b) = (xy, ab)$,
and thus $\mathbf{H}^{[\mathcal{S}]}\cong \mathbf{H}\times\mathbf{S}$. 
\end{example}

\begin{example}\label{lzero}
Let $\mathbf{1}$ be the trivial semigroup, and let 
$I = \{0,1\}$. Next, let $\lambda\colon I\to I$ be the identity map,
and let $\rho\colon I\to I$ be the constant map $\overline{0}$.
This defines a $\lambda\rho$-system $\mathcal{I}$ over $\mathbf{1}$.
Consider $\mathbb{Z}_2^{[\mathcal{I}]}$, whose universe
$\mathbb{Z}_2^I$ we will identify in the obvious way with
the set $\{00,01,10,11\}$. Here is the multiplication table of
$\mathbb{Z}_2^{[\mathcal{I}]}$:
$$
\begin{tabular}{c|cccc}
$\star$ & $00$  & $11$ & $01$ & $10$ \\
\hline                            
$00$  & $00$  & $11$ & $00$ & $11$ \\
$11$  & $11$  & $00$ & $11$ & $00$ \\  
$01$  & $01$  & $10$ & $01$ & $10$ \\
$10$  & $10$  & $01$ & $10$ & $01$ \\
\end{tabular}  
$$
Partitioning the universe into $\{00,11\}$ and $\{01,10\}$, we
obtain a congruence $\theta$, such that 
$\mathbb{Z}_2^{[\mathcal{I}]}/\theta$ is isomorphic to the two-element left-zero semigroup.
\end{example}

\begin{example}\label{flip-flop}
Let $\mathbf{2}=(\{0,1\},\vee)$ be the two-element join-semilattice, and let
$\mathcal{Z}$ be the $\lambda\rho$-system over $\mathbf{2}$, defined by putting
\begin{enumerate}
\item $I[0]=\{0\},$ $I[1]=\{0,1\}$,
\item $\lambda[1,0]=\rho[0,1]=\lambda[1,1]=id_{I[1]}$ and
  $\rho[1,1] = \overline{0}$. 
\end{enumerate}
This defines a unique $\lambda\rho$-system, since 
the remaining maps all have range $\{0\}$. It is
easy to show that the semigroup $\mathbb Z_2^{[\mathcal{Z}]}$ is the following:
$$
\begin{array}{c|cccccc}
\star &0&1&00&11&01&10\\
\hline
0&0&1&00&11&01&10\\
1&1&0&11&00&10&01\\
00&00&11&00&11&00&11\\
11&11&00&11&00&11&00\\
01&01&10&01&10&01&10\\
10&10&01&10&01&10&01
\end{array}
$$
Partitioning the universe into $\{0,1\}$, $\{00,11\}$ and $\{01,10\}$
we obtain a congruence $\theta$, such that $\mathbb Z_2^{[\mathcal{Z}]}/\theta$ is
isomorphic to the left flip-flop monoid $L_2^1$. 
\end{example}

In the commonly used terminology, Examples~\ref{lzero} and~\ref{flip-flop} show,
respectively, that the two-element left-zero semigroup $L_2$ strongly divides
a $\lambda\rho$-product of $\mathbb{Z}_2$ over the trivial semigroup, and
the three-element left flip-flop monoid $L_2^1$ strongly divides 
a $\lambda\rho$-product of $\mathbb{Z}_2$ over a two-element semilattice.
In this sense, $L_2^1$ turns out to be decomposable.

The next proposition shows that $\lambda\rho$-products generalise
wreath products and block products. As there
are a number of slightly different versions of wreath products and block
products for semigroups, we will state the definitions we use.

Let a semigroup $\mathbf{S}$ act on a set $X$ on the right, and let $\mathbf{H}$
be any semigroup. The wreath product $\mathbf{H}\wr (X,\mathbf{S})$ is
a semidirect product $\mathbf{H}^X\rtimes  \mathbf{S}$ with multiplication
defined by $(u, a)\star(w,b) = (u\cdot (w\circ(\bl \ast a)),\ ab)$, where
$\cdot$ is the multiplication in $\mathbf{H}^X$ and $\ast$ is the right action
of $\mathbf{S}$ on $X$.

A two-sided action of a 
semigroup $\mathbf{S}$ on a set $X$ is a pair of maps $\ld\colon S\times X\to X$ and
$\rd\colon X\times S\to X$, satisfying 
$$
a\ld (b\ld x) = (a\cdot b)\ld x\qquad (x\rd a)\rd b = x\rd (a\cdot b)\qquad
(a\ld x)\rd b=a\ld (x\rd b)
$$
for any $a,b\in S$  and $x\in X$. The two-sided wreath product
$\mathbf{H}\wr (X,\mathbf{S},X)$ is a semidirect product
$\mathbf{H}^X\bowtie \mathbf{S}$, with multiplication defined by
$(u, a)\star(w,b) = ((u\circ(b\ld\bl)) \cdot (w\circ(\bl\rd a)),\ ab)$.

Any semigroup $\mathbf{N}$ has a natural
two-sided action on $N^2$, given by $n\ld(n_1,n_2) = (nn_1, n_2)$ and
$(n_1,n_2)\rd n = (n_1,n_2n)$. The block product
$\mathbf{H}\mathbin{\Box}\mathbf{N}$ is the two-sided wreath product
$\mathbf{H}\wr (N^2,\mathbf{N},N^2)$ with
respect to the natural two-sided action of $\mathbf{N}$ on $N^2$.

\begin{proposition}\label{action-wreath-block}
Let $(X,\ld,\rd,\mathbf{S})$ consist of a set
$X$ together with a two-sided 
action of a semigroup $\mathbf{S}$ on $X$. Then the system of maps
$$
\mathcal{S}(X,\mathbf{S},X) =  \bigl(\langle \lambda[a,b],\rho[a,b]\rangle\colon
I[ab]\to I[a]\times I[b]\bigr),
$$
where $I[s]=X$ for any $s\in S$, and 
\begin{enumerate}
\item $\lambda[a,b] = b\ld\bl$\; for any $a,b\in S$,
\item $\rho[a,b] = \bl\rd a$\; for all $a,b\in S$. 
\end{enumerate}
is a $\lambda\rho$-system over $\mathbf{S}$. Moreover, for any semigroup
$\mathbf{H}$, the $\lambda\rho$-product
$\mathbf{H}^{[\mathcal{S}(X,\mathbf{S},X)]}$ is isomorphic 
to the two-sided wreath product of\/ $\mathbf{H}$ by
$\mathbf{S}$
\end{proposition}

Taking $\ld$ to be the second projection,
and $\rd$ to be the right action of $\mathbf{S}$ on $X$,
we get that $\mathbf{H}^{[\mathcal{S}(X,\mathbf{S},X)]}$ is
isomorphic to the wreath product $\mathbf{H}\wr (X,\mathbf{S})$.
In the more usual notation, with $\ast$ replacing $\rd$,
the explicit definitions of $\lambda$ and $\rho$ become   
\begin{enumerate}
\item $\lambda[a,b] = {id}_X$ for all $a,b\in S$,
\item $\rho[a,b] = \bl\ast a$ for all $a,b\in S$. 
\end{enumerate}
We will denote the
resulting $\lambda\rho$-system by $\mathcal{S}(X,\mathbf{S})$.
It will be particularly important in Section~\ref{wreath}.
(Analogously, we can define a $\lambda\rho$-system
$\mathcal{S}(\mathbf{S}, X)$ starting from a left action of $\mathbf{S}$ on $X$.) 

Taking $(S^2,\ld,\rd,\mathbf{S})$ to be the natural two-sided action
of $\mathbf{S}$ on $S^2$ we get that
$\mathbf{H}^{[\mathcal{S}(S^2,\mathbf{S},S^2)]}$ is isomorphic  
to the block product $\mathbf{H}\mathbin{\Box}\mathbf{S}$.

Note that in Proposition~\ref{action-wreath-block}
letting $\mathbf{S}$ be the trivial semigroup 
and $|X| = 1$ results in $\mathbf{H}^{[\mathcal{S}]}$
being isomorphic to $\mathbf{H}$, for any semigroup
$\mathbf{H}$. Example~\ref{lzero} shows that this is not the case 
for arbitrary $\lambda\rho$-products. Example~\ref{flip-flop} shows that there
are nontrivial $\lambda\rho$-products not isomorphic to nontrivial wreath
products (otherwise the flip-flop monoid would divide a nontrivial wreath product).
In general, there are nontrivial $\lambda\rho$-products not isomorphic to any
nontrivial semidirect products, as the next example shows.

\begin{example}\label{not-semidir-example}
Let $\underline{\mathbf{2}}=(\{0,1\},\wedge)$ be the two-element
meet-semilattice, and let 
$\mathcal{U}$ be the $\lambda\rho$-system over $\underline{\mathbf{2}}$, defined
by putting 
\begin{enumerate}
\item $I[0]=\emptyset$, $I[1]=\{0,1\}$,
\item $\lambda[1,1]=\rho[1,1]=id_{I[1]}$. 
\end{enumerate}
All other maps have $I[0]$ as the domain, so they are empty maps. It is easy to
check that this defines a unique $\lambda\rho$-system. Let $\mathbf{S}$ be any
two-element semigroup. Then the universe of $\mathbf{S}^{[\mathcal{U}]}$
is $S^2 \uplus S^\emptyset$ so it has 5 elements, and hence cannot be
the universe of any nontrivial semidirect product.
\end{example}
This example adds zero to $\mathbf{S}^2$ and it can be easily modified in various
ways. Taking $I[1]$ to be a singleton and an arbitrary $\mathbf{S}$ shows that
adding zero in general can be viewed as a $\lambda\rho$-product.

In contrast to the above, we will show that in the case of groups
$\lambda\rho$-products coincide with the usual wreath
products. To be precise, in the final section we show that for any $\lambda\rho$-system
$\mathcal{S}$ over a group $\mathbf{G}$, if the $\lambda\rho$-product
$\mathbf{H}^{[\mathcal{S}]}$ for any group $\mathbf{H}$ is itself a group, then
$\mathbf{H}^{[\mathcal{S}]}$ is isomorphic to a wreath product
$\mathbf{H}\wr (X,\mathbf{G})$, with $\mathbf{G}$ acting on some set $X$.

The next example comes from the theory of residuated structures.
We present it mainly because it is somewhat related to the 
kite construction that motivated the present work. The reader unfamiliar with
residuated structures can safely skip the example. The reader familiar with
residuated structures will notice that the slashes used here are the opposites of
the slashes used for the two-sided action in
Proposition~\ref{action-wreath-block}.

\begin{example}\label{rl-example}
Let $(L;\leq, \cdot,\ld,\rd)$ be a partially ordered residuated semigroup.
For each $a\in L$ we put 
$I[a] = \{u\in L\colon a\leq u\}$. Next, let $\lambda[a,b] = \bl\rd b$ and
$\rho[a,b] = a\ld\bl$. Then $(\mathbf{I},\blambda,\brho)$ is a
$\lambda\rho$-system over $(L;\cdot)$.
\end{example}

The last example in this section is hardly more than a curiosity, but we find it
quite illustrative. Let $\bullet$ be any semigroup operation on a
two-element Boolean algebra $\mathbf{B}$, say, meet, join, projection, or
addition modulo 2. Then, for any set $X$, on the one hand
$\bullet$ is a pointwise operation in $\mathbf{B}^X$, but on the
other hand, it has its \emph{alter ego} in the powerset $2^X$, via
characteristic functions. Here is an analogue of this for a
$\lambda\rho$-system over $\mathbf{B}$.
\begin{example}
Let $\mathcal{S} = (\mathbf{I},\blambda,\brho)$ be any
$\lambda\rho$-system over a semigroup $\mathbf{S}$. Let $\bullet$ be any semigroup
operation on the two-element Boolean algebra $\mathbf{B}$. Then, 
$\mathbf{B}^{[\mathcal{S}]}$ is a semigroup whose universe is 
$\biguplus\{2^{I[a]}: a\in S\}$. The semigroup operation  
can be explicitly written as
$$
(U,a)\star (W,b) = \bigl(\lambda[a,b]^{-1}(U)\bullet\rho[a,b]^{-1}(W), ab\bigr) 
$$
where $U\subseteq I[a]$ and $W\subseteq I[b]$.
\end{example}  
One may think of the preimages $\lambda[a,b]^{-1}(U)$ and $\rho[a,b]^{-1}(W)$
as shadows cast by $U$ and $W$ in a stack of Venn diagrams.

\section{An application of Grothendieck construction}\label{cats}

In this short section we use some categorical tools to show that general
$\lambda\rho$-systems form a category in a natural way. Of itself, it does not add
anything essentially new to the construction of $\lambda\rho$-systems,
it just provides a conceptualisation which
will be useful in Section~\ref{free-constr}, but perhaps
it may also prove useful in developing the theory further.
Throughout this section $\mathsf{Cat}$ will stand for the category of all
categories (with functors as arrows). For any
category $\mathsf{C}$, we will write $\mathrm{obj}(\mathsf{C})$ for the class of
objects of $\mathsf{C}$.  

\begin{definition}\label{t-over-S}
Let $\mathcal{S} = (\mathbf{I},\blambda,\brho)$ and
$\mathcal{S}'=(\mathbf{I}',\blambda',\brho')$  be
$\lambda\rho$-systems over a semigroup $\mathbf S$. We define
a \emph{morphism of $\lambda\rho$-systems}
$\mathbf{t}$ from $\mathcal{S}$ to $\mathcal{S}'$
to be a system of maps 
$\mathbf{t} = (t[a]\colon I[a]\longrightarrow I'[a])_{a\in S}$
satisfying $\lambda'[a,b]\circ t[ab] = t[a]\circ \lambda[a,b]$ and
$\rho'[a,b]\circ t[ab]= t[b]\circ \rho[a,b]$ for all $a,b\in S$, i.e., such that
the diagrams below commute. 
$$
\begin{tikzpicture}[>=stealth,auto]
\node (tab) at (0,0) {$I[ab]$};
\node (ab) at (3,0) {$I'[ab]$};
\node (ta) at (0,-2) {$I[a]$};
\node (a) at (3,-2) {$I'[a]$};
\draw[->] (tab) to node {$t[ab]$} (ab);
\draw[->] (tab) to node[swap] {$\lambda[a,b]$} (ta);
\draw[->] (ab) to node[swap] {$\lambda'[a,b]$} (a);
\draw[->] (ta) to node {$t[a]$} (a);
\end{tikzpicture} 
\qquad
\begin{tikzpicture}[>=stealth,auto]
\node (tab) at (0,0) {$I[ab]$};
\node (ab) at (3,0) {$I'[ab]$};
\node (tb) at (0,-2) {$I[b]$};
\node (b) at (3,-2) {$I'[b]$};
\draw[->] (tab) to node {$t[ab]$} (ab);
\draw[->] (tab) to node[swap] {$\rho[a,b]$} (tb);
\draw[->] (ab) to node[swap] {$\rho'[a,b]$} (b);
\draw[->] (tb) to node {$t[b]$} (b);
\end{tikzpicture} 
$$
\end{definition}

It is easy to see that $\lambda\rho$-systems over a semigroup $\mathbf{S}$ form
a category whose arrows are morphisms of $\lambda\rho$-systems.
The composition of morphisms of $\lambda\rho$-systems
(defined naturally as a system of compositions of maps) is a
morphism of $\lambda\rho$-systems, and the identity arrow is a system of identity maps.

\begin{definition}\label{cat-lr-over-S}
Let $\mathbf{S}$ be a semigroup. We define $\bm{\lambda\rho}(\mathbf{S})$ to be  
the category whose objects are $\lambda\rho$-systems over a semigroup
$\mathbf{S}$, and whose arrows are morphisms of $\lambda\rho$-systems.
\end{definition}

Having defined the category of $\lambda\rho$-systems
over a fixed semigroup, we will upgrade this definition to general 
$\lambda\rho$-systems. We will do it by means of
Grothendieck construction, whose one version we will now recall.

\begin{definition}[Grothendieck construction]
Let\/ $\mathsf{C}$ be an arbitrary category, and let
$F\colon \mathsf{C}^{op}\to\mathsf{Cat}$
be a functor. Then, $\mathsf{\Gamma}(F)$ is the category defined as follows.
\begin{enumerate}
\item Objects of\/ $\mathsf{\Gamma}(F)$ are pairs $(A,X)$ such that
$A\in \mathrm{obj}(\mathsf{C})$ and $X\in \mathrm{obj}(F(A))$. 
\item Arrows between objects $(A_1,X_1),(A_2,X_2)\in
\mathrm{obj}(\mathsf{\Gamma}(F))$ are pairs
  $(f,g)$ such that $f\colon A_2\to A_1$ is an arrow in the category
$\mathsf{C}$ and $g\colon F(f)(X_1)\to X_2$ is an arrow in the category $F(A_2)$. 
\item For objects and arrows in $\mathsf{\Gamma}(F)$, given below:
$$
(A_1,X_1)\overset{(f_1,g_1)}\longrightarrow(A_2,X_2)
\overset{(f_2,g_2)}\longrightarrow(A_3,X_3) 
$$ 
the composition of arrows is defined by:
$$
(f_2,g_2)\circ(f_1,g_1)= (f_1\circ f_2, g_2\circ F(f_2)(g_1)). 
$$
\end{enumerate}
\end{definition}

To apply Grothendieck construction to $\lambda\rho$-systems, we first show the
existence of a suitable contravariant functor from semigroups to categories.

\begin{lemma}
Let\/ $\mathsf{Sg}$ be the category of semigroups
\textup{(}with homomorphisms\textup{)}.  
There is a functor $\bm{\lambda\rho}\colon \mathsf{Sg}^{op}\to
\mathsf{Cat}$ such that $\mathbf S\mapsto \bm{\lambda\rho}(\mathbf S)$, and
for each semigroup homomorphism $f\colon \mathbf T\to\mathbf S$
we have a functor 
$$
\bm{\lambda\rho}(f)\colon \bm{\lambda\rho}(\mathbf S)\longrightarrow
\bm{\lambda\rho}(\mathbf T)
$$
such that 
\begin{enumerate}
\item If $\mathcal{S} = (\mathbf{I},\blambda,\brho)\in
\bm{\lambda\rho}(\mathbf S)$, then
$$
\bm{\lambda\rho}(f)(\mathcal{S}) =
(\bm{\lambda\rho}(f)\mathbf{I},\bm{\lambda\rho}(f)\blambda,
\bm{\lambda\rho}(f)\brho)
$$ 
where 
\begin{eqnarray*}
\bm{\lambda\rho}(f)\mathbf{I} &=& \bigl(I[f(a)]\bigr)_{a\in T},\\
\bm{\lambda\rho}(f)\blambda &=&
\bigl(\lambda[f(a),f(b)]\colon
I[f(ab)]\to I[f(a)]\bigr)_{(a,b)\in T\times T},\\
\bm{\lambda\rho}(f)\brho &=&
\bigl(\rho[f(a),f(b)]\colon
I[f(ab)]\to I[f(b)]\bigr)_{(a,b)\in T\times T}.
\end{eqnarray*}
\item For any $\lambda\rho$-systems $\mathcal{S} = (\mathbf{I},\blambda,\brho)$ and
$\mathcal{S}' = (\mathbf{I}',\blambda',\brho')$ over a semigroup
$\mathbf{S}$, and for any morphism of $\lambda\rho$-systems
$\mathbf{t}\colon \mathcal{S}\to\mathcal{S}'$, such that 
$$
\mathbf{t} = \bigl(t[a]\colon I[a]\longrightarrow I'[a]\bigr)_{a\in S}
$$
we have a morphism of $\lambda\rho$-systems
$\bm{\lambda\rho}(f)\mathbf{t}\colon
\bm{\lambda\rho}(f)(\mathcal{S})\to \bm{\lambda\rho}(f)(\mathcal{S}')$
such that
$$
\bm{\lambda\rho}(f)\mathbf{t} =
\bigl(t[f(a)]\colon I[f(a)]\longrightarrow I'[f(a)]\bigr)_{a\in T}.
$$
\end{enumerate}
\end{lemma}
\begin{proof}
The proof is a series of tedious but straightforward calculations, which we
omit. A crucial point is that since $\bm{\lambda\rho}(f)$ acts contravariantly,
$\bm{\lambda\rho}(f)\mathbf{I}$, 
$\bm{\lambda\rho}(f)\blambda$ and $\bm{\lambda\rho}(f)\brho$ are 
well defined. For the proofs that ($\alpha$), ($\beta$) and ($\gamma$) are
satisfied, and that $\bm{\lambda\rho}(f)$ behaves properly on morphisms of
$\lambda\rho$-systems, we only need the definitions and the fact that $f$ is a
homomorphism. 
\end{proof}

Applying Grothendieck construction with
$\mathsf{C} = \mathsf{Sg}$ and
$F = \bm{\lambda\rho}$, we obtain
a category $\mathsf{\Gamma}(\bm{\lambda\rho})$ of general
$\lambda\rho$-systems. The next lemma characterises the arrows of this category.

\begin{lemma}\label{groth-arrows}
Let $(\mathbf{S}, \mathcal{S})$ and $(\mathbf{T}, \mathcal{T})$
be general $\lambda\rho$-systems, with
$\mathcal{S} =  (\mathbf{I}, \boldsymbol{\lambda}^I, 
\boldsymbol{\rho}^I)$ and $\mathcal{T} = (\mathbf{J}, 
\boldsymbol{\lambda}^J, \boldsymbol{\rho}^J)$.
An arrow from $(\mathbf{S}, \mathcal{S})$ to $(\mathbf{T},\mathcal{T})$ 
is a pair $(h, \mathbf{t})$ consisting  
of a homomorphism $h\colon \mathbf{T}\to \mathbf{S}$ and 
a system of maps $\mathbf{t} = \bigl(t[a]\colon I[{h(a)}] \to J[a]\bigr)_{a\in T}$
such that the diagrams below commute.
$$
\begin{tikzpicture}[>=stealth,auto]
\node (tab) at (0,0) {$I[{h(a)h(b)}] = I[{h(ab)}]$};
\node (ab) at (3,0) {$J[{ab}]$};
\node (ta) at (0,-2) {$I[{h(a)}]$};
\node (a) at (3,-2) {$J[a]$};
\draw[->] (tab) to node {$t[ab]$} (ab);
\draw[->] (tab) to node[swap] {$\lambda^I[{h(a),h(b)}]$} (ta);
\draw[->] (ab) to node[swap] {$\lambda^J[{a,b}]$} (a);
\draw[->] (ta) to node {$t[a]$} (a);
\end{tikzpicture} 
\qquad
\begin{tikzpicture}[>=stealth,auto]
\node (tab) at (0,0) {$I[{h(a)h(b)}] = I[{h(ab)}]$};
\node (ab) at (3,0) {$J[{ab}]$};
\node (tb) at (0,-2) {$I[{h(a)}]$};
\node (b) at (3,-2) {$J[b]$};
\draw[->] (tab) to node {$t[ab]$} (ab);
\draw[->] (tab) to node[swap] {$\rho^I[{h(a),h(b)}]$} (tb);
\draw[->] (ab) to node[swap] {$\rho^J[{a,b}]$} (b);
\draw[->] (tb) to node {$t[b]$} (b);
\end{tikzpicture} 
$$
\end{lemma}  

\begin{proof}
Immediate from Grothendieck construction.
\end{proof}  

We will refer to arrows of $\mathsf{\Gamma}(\bm{\lambda\rho})$ as
\emph{transformations}. Morphisms of $\lambda\rho$-systems are then a particular case
of transformations. Namely, for general $\lambda\rho$-systems
$(\mathbf{S}, \mathcal{S})$ and $(\mathbf{T}, \mathcal{T})$,
if $\mathbf{S} = \mathbf{T}$, then
for any transformation $(id_S,\mathbf{t})\colon \mathcal{T}\to\mathcal{S}$
we have that $\mathbf{t}$ is a morphism of $\lambda\rho$-systems.

Any $\lambda\rho$-system over a semigroup
$\mathbf{S}$ has a natural restriction to any subsemigroup $\mathbf{T}$ of
$\mathbf{S}$.  
Let $\mathcal{S} = (\mathbf{I}, \boldsymbol{\lambda},\boldsymbol{\rho})$
be a $\lambda\rho$-system over $\mathbf{S}$  and\/ let $\mathbf{T}\leq\mathbf{S}$. Then
$\mathcal{T} = (\mathbf{I}|_T,\blambda|_T,\brho|_T)$, where
$\mathbf{I}|_T$, $\blambda|_T$ and $\brho|_T$ are the restrictions of\/
$\mathbf{I}$, $\blambda$ and $\brho$ to $T$,  
is a $\lambda\rho$-system over $\mathbf{T}$. Moreover,
$(e,\mathbf{t})\colon (\mathbf{S},\mathcal{S})\to(\mathbf{T},\mathcal{T})$, defined by
taking $e\colon\mathbf{T}\to\mathbf{S}$ to be the identity embedding,
and $\mathbf{t} = \bigl(t[a]\colon I[e(a)] \to I[a]\bigr)_{a\in T}$,
where $t[a] = {id}_{I[a]}$, is obviously a transformation.
Note that restrictions are completely determined by subsemigroups, so 
we may write $(\mathbf{T}, \mathcal{S}|_T)$ for a restriction
of $(\mathbf{S},\mathcal{S})$ with $\mathbf{T}\leq\mathbf{S}$.

As usual, we will write
$(\mathbf{S},\mathcal{S})\cong(\mathbf{T},\mathcal{T})$
for general $\lambda\rho$-systems isomorphic in the category
$\mathsf{\Gamma}(\boldsymbol{\lambda\rho})$.
In Section~\ref{wreath} it will be useful to have 
a more explicit characterisation of isomorphic 
general $\lambda\rho$-systems, which is given below without an easy proof.  

\begin{lemma}\label{isom}
Let $(\mathbf{S},\mathcal{S})$ and $(\mathbf{T},\mathcal{T})$ be general
$\lambda\rho$-systems. $(\mathbf{S},\mathcal{S})$ and
$(\mathbf{T},\mathcal{T})$ are isomorphic if and only if 
there exists a transformation $(e,\mathbf{t})\colon
(\mathbf{S},\mathcal{S})\to(\mathbf{T},\mathcal{T})$ such that
$e\colon\mathbf{T}\to\mathbf{S}$ is an isomorphism of semigroups, and each
$t[a]$ in the system  
$\mathbf{t} = \bigl(t[a]\colon I[e(a)] \to I[a]\bigr)_{a\in T}$
is a bijection.
\end{lemma}

\section{Construction of $\lambda\rho$-systems over free
  semigroups}\label{free-constr} 

It this section we show that any family of sets $F = \{I[x]: x\in X\}$
gives rise to a natural $\lambda\rho$-system $\mathcal{X}^+$ over the free
semigroup $X^+$, which
is in a good sense the most general
$\lambda\rho$-system associated with $F$. For if $X$ happens to be the
universe of a semigroup $\mathbf{X}$, and $F$ carries the
structure of a $\lambda\rho$-system $\mathcal{F}$ over $\mathbf{X}$, then there is a
transformation from $\mathcal{F}$ to $\mathcal{X}^+$, and moreover for any
semigroup $\mathbf{H}$ there is an onto homomorphism from
$\mathbf{H}^{[\mathcal{X}^+]}$ to $\mathbf{H}^{[\mathcal{F}]}$.   

\begin{definition}\label{over-free-sgrp}
Let $X$ be a nonempty set, and let $I[x]$ be a set for each $x\in X$.
Let $X^+$ be the free semigroup, freely generated by some set $X$,
\begin{itemize}
\item For each word $w = x_1x_2\cdots x_k\in X^+$,   
we put $I[w] = I[{x_1}] \times \dots \times I[{x_k}]$.
\item Since $I[wu] = I[w]\times I[u]$ for all $w,u\in X^+$,  we put
\begin{itemize}  
\item $\lambda[w,u]\colon I[wu] \to I[w]$ to be the first projection and
\item $\rho[w,u]\colon I[wu] \to I[u]$ to be the second projection.
\end{itemize}
\end{itemize}
\end{definition}

\begin{lemma}\label{free-is-free}
Let\/  $X^+$ be the free semigroup generated by $X$, and let
$$
\mathcal{X}^+ = \bigl(\langle \lambda[w,u],\rho[w,u]\rangle\colon
I[wu]\longrightarrow I[w]\times I[u]\bigr)_{(w,u)\in (X^+)^2}
$$
be the system of sets and maps of Definition~\ref{over-free-sgrp}. Then
$\mathcal{X}^+$ is a $\lambda\rho$-system over $X^+$.
\end{lemma}

\begin{proof}
The commutation conditions ($\alpha$), ($\beta$) and ($\gamma$) 
clearly hold as the maps are compositions of projections. 
\end{proof}

Any transformation of $\lambda\rho$-systems gives rise
to a homomorphism of $\lambda\rho$-products.

\begin{definition}\label{hom-from-trans}
Let $(\mathbf{S},\mathcal{S})$ and
$(\mathbf{T},\mathcal{T})$ be general $\lambda\rho$-systems, with
$\mathcal{S} =  (\mathbf{I}, \boldsymbol{\lambda}^I, 
\boldsymbol{\rho}^I)$ and $\mathcal{T} = (\mathbf{J}, 
\boldsymbol{\lambda}^J, \boldsymbol{\rho}^J)$.
Let $(h, \mathbf{t})\colon (\mathbf{S},\mathcal{S})\to
(\mathbf{T},\mathcal{T})$ be a transformation.
For any semigroup $\mathbf{H}$, we define
$\mathbf{H}^{(h,\mathbf{t})}\colon 
\mathbf{H}^{[\mathcal{T}]} \to \mathbf{H}^{[\mathcal{S}]}$ to be the map 
$$
\mathbf{H}^{(h, \mathbf{t})}(x,a) = (x\circ t[a],\ h(a))
$$
for every $(x,a)\in\biguplus_{a\in T} H^{J[a]}$. 
\end{definition}

The notation $\mathbf{H}^{(h,\mathbf{t})}$, common in category theory 
unfortunately produces a slight notational clash. The map applies to an element
$(x,a)$, where the second coordinate is from $\mathbf{T}$,
but in the superscript we have  $(h,\mathbf{t})$, where the first coordinate is
a homomorphism from $\mathbf{T}$ to $\mathbf{S}$. This is done for consistency
with general $\lambda\rho$-systems on the one hand and
transformations on the other, and should not cause confusion.

\begin{theorem}\label{trans}
Let $(\mathbf{S},\mathcal{S})$, $(\mathbf{T},\mathcal{T})$ and
$(h, \mathbf{t})\colon (\mathbf{S},\mathcal{S})\to (\mathbf{T},\mathcal{T})$
be as above. Then,
$\mathbf{H}^{(h,\mathbf{t})}\colon 
\mathbf{H}^{[\mathcal{T}]} \to \mathbf{H}^{[\mathcal{S}]}$ defined above
is a homomorphism for any semigroup $\mathbf{H}$.
Moreover, $\mathbf{H}^{-}$ is a contravariant functor from the category
$\Gamma(\blambda\brho)$ to the category $\mathsf{Sg}$ of semigroups.
\end{theorem}

\begin{proof}
It is clear that the map $\mathbf{H}^{(h, \mathbf{t})}$ is well defined.
Let $(x,a), (y,b)\in \biguplus_{a\in T} H^{J[a]}$. Then, we have
\begin{align*}
\mathbf{H}^{(h, \mathbf{t})}\bigl((x,a)\star (y,b)\bigr)
&=\mathbf{H}^{(h,\mathbf{t})}\bigl((x\circ\lambda^J[{a,b}])
    (y\circ\rho^J[{a,b}]),\ ab\bigr)\\ 
&= \Bigl(\bigl((x\circ\lambda^J[{a,b}])(y\circ\rho^J[{a,b}])\bigr)\circ t[ab],\
                                          h(ab)\Bigr)\\   
&= \Bigl((x\circ\lambda^J[{a,b}]\circ t[ab])(y\circ\rho^J[{a,b}]\circ t[ab]),\
h(a)h(b)\Bigr) \\
&= \Bigl(\bigl(x\circ t[a]\circ \lambda^I[{h(a),h(b)}]\bigr)
    \bigl(y\circ t[b]\circ\rho^I[{h(a),h(b)}]\bigr),\ h(a)h(b)\Bigr)\\
&= \bigl(x\circ t[a],\ h(a)\bigr)\star\bigl(y\circ t[b],\ h(b)\bigr)\\
&= \mathbf{H}^{(h, \mathbf{t})}(x,a)\star \mathbf{H}^{(h, \mathbf{t})}(y,b). 
\end{align*}
The proof of the moreover part is straightforward. 
\end{proof}

Now, consider a $\lambda\rho$-system $\mathcal{S}$ over some semigroup 
$\mathbf{S}$. Taking $S$ as the set of free generators, form 
the free semigroup $S^+$. Let $\otimes\colon S^+\to \mathbf{S}$ 
be the homomorphism extending the identity map
on $S$, so that $\otimes s = s$ for any $s\in S$.
We will write
$\otimes(s_1s_2\dots s_n)$ for the product of the elements
$s_1,s_2,\dots, s_n$ of $S$ in $\mathbf{S}$, reserving
$s_1s_2\dots s_n$ for the word in $S^+$.

\begin{definition}\label{free-t}
Let $\mathbf{S}$, $\mathcal{S}$ and $\otimes$ be as above, and let
$\mathcal{S}^+$ be the $\lambda\rho$-system over $S^+$ of
Definition~\ref{over-free-sgrp}. We define a system of maps 
$$
\mathbf{t} = (t[w]\colon I[\otimes w] \to
I[{s_1}]\times I[{s_2}]\times\dots\times I[{s_n}])_{w\in S^+}
$$
where $w = s_1s_2\dots s_n$, as follows. For each $s\in S$ we put
$t[s]\colon I[\otimes s]\to I[s]$ to be the identity map on $I[s]$.  
For each $w = s_1s_2\dots s_n$ with $n\geq 2$, and each $z\in I[\otimes w]$ we put
$t[s_1s_2\dots s_n](z) = (v_1,\dots, v_n)$, where
\begin{align*}
v_1 &= \lambda[s_1, \otimes(s_2s_3\cdots s_n)](z),\\
v_j &= \rho[\otimes(s_1\cdots s_{j-1}), s_j]\circ 
\lambda[\otimes(s_1\cdots s_j), \otimes(s_{j+1}\cdots s_n)](z),\\
&\qquad\text{ for each } j\in\{2,\dots,n-1\},\\  
v_n &= \rho[\otimes(s_1\cdots s_{n-1}), s_n](z),
\end{align*}
and $\lambda$, $\rho$ are from $\mathcal{S}$. 
\end{definition}

If $I[\otimes w] =\emptyset$ then
$t[w]$ is the empty map, of course.
  
\begin{lemma}\label{well-defd}
Let $\mathbf{S}$, $\mathcal{S}$, $\mathcal{S}^+$, $\otimes$ and\/
$\mathbf{t}$ be as in Definition~\ref{free-t}.   Then, the following hold:
\begin{enumerate}
\item For each $s\in S$, we have $t[s] = {id}_{I[s]}$.
\item For any $s_1,s_2,\dots, s_n \in S$ and any $z\in I[\otimes(s_1s_2\dots s_n)]$
we have
\begin{align*}
v_j &= \rho[\otimes(s_1\cdots s_{j-1}), s_j]\circ 
\lambda[\otimes(s_1\cdots s_j), \otimes(s_{j+1}\cdots s_n)](z)\\
&= \lambda[s_j, \otimes(s_{j+1}\cdots s_n)]\circ 
\rho[\otimes(s_1\cdots s_{j-1}),\otimes(s_{j}\cdots s_n)](z)
\end{align*}
for each $j\in \{2,\cdots,n-1\}$.
\end{enumerate}
\end{lemma}

\begin{proof}
We have (1) directly from Definition~\ref{free-t}, and  
(2) follows easily from the fact that $\lambda$ and $\rho$ come from
$\mathcal{S}$: the non-definitional equality is an application of ($\gamma$).  
\end{proof}

\begin{lemma}
Let $\mathbf{S}$, $\mathcal{S}$, $\mathcal{S}^+$, $\otimes$ and\/
$\mathbf{t}$ be as in Definition~\ref{free-t}.  
Then, $(\otimes,\mathbf{t})\colon (\mathbf{S},\mathcal{S})
\to \bigl(S^+,\mathcal{S}^+\bigr)$ is a transformation. 
\end{lemma}

\begin{proof}
It is clear that the range of each map $t[s_1s_2\cdots s_n]$ belongs 
to  $I[{s_1s_2\cdots s_n}]$. We need to show that the following diagrams commute
$$
\begin{tikzpicture}[>=stealth,auto]
\node (tab) at (0,0) {$I[{\otimes(wu)}]$};
\node (ab) at (3,0) {$I[w]\times I[u]$};
\node (ta) at (0,-2) {$I[{\otimes w}]$};
\node (a) at (3,-2) {$I[w]$};
\draw[->] (tab) to node {$t[wu]$} (ab);
\draw[->] (tab) to node[swap] {$\lambda[{\otimes w,\otimes u}]$} (ta);
\draw[->] (ab) to node[swap] {$\lambda[{w,u}]$} (a);
\draw[->] (ta) to node {$t[w]$} (a);
\end{tikzpicture} 
\qquad
\begin{tikzpicture}[>=stealth,auto]
\node (tab) at (0,0) {$I[{\otimes(wu)}]$};
\node (ab) at (3,0) {$I[w]\times I[u]$};
\node (tb) at (0,-2) {$I[{\otimes w}]$};
\node (b) at (3,-2) {$I[u]$};
\draw[->] (tab) to node {$t[wu]$} (ab);
\draw[->] (tab) to node[swap] {$\rho[{\otimes w,\otimes u}]$} (tb);
\draw[->] (ab) to node[swap] {$\rho[{w,u}]$} (b);
\draw[->] (tb) to node {$t[u]$} (b);
\end{tikzpicture} 
$$
where $w = s_1\dots s_k\in S^+$ and $u = s_{k+1}\cdots s_n\in S^+$.
We can assume that $I[\otimes(s_1\dots s_n)]\neq\emptyset$. 
Let $z\in I[\otimes(wu)]$ and   
let $y = \lambda[\otimes w,\otimes u](z)$.
Consider the left diagram. Let
$t[wu](z) = (v_1,\dots,v_k,v_{k+1},\dots,v_n)$ and 
$t[w](y) = (v'_1,\dots,v'_k)$. 
Since $\lambda[w,u]$ is the projection onto $I[w]$ we have
$\lambda[w,u]\circ t[wu](z) = (v_1,\dots,v_k)$, so we need to
verify that $(v_1,\dots,v_k) = (v_1',\dots,v_k')$.
By definition of the maps $t$ we have
\begin{align*}
v_1 &= \lambda[s_1,\otimes(s_2\cdots s_n)](z)\\
&= \lambda[s_1,\otimes(s_2\cdots s_k)]\circ
\lambda[\otimes(s_1\cdots s_k),\otimes(s_{k+1}\cdots s_n)](z)\\
&= \lambda[s_1,\otimes(s_2\cdots s_k)]\circ\lambda[\otimes w,\otimes u](z) \\
&= \lambda[s_1,\otimes(s_2\cdots s_k)](y)\\
&= v'_1
\end{align*}
then, for $j\in\{2,\dots,k-1\}$
\begin{align*}
v_j &= \rho[\otimes(s_1\cdots s_{j-1}),s_j]\circ 
\lambda[\otimes(s_1\cdots s_j),\otimes(s_{j+1}\cdots s_n)](z)\\
&= \rho[\otimes(s_1\cdots s_{j-1}),s_j]\circ 
\lambda[\otimes(s_1\cdots s_j),\otimes(s_{j+1}\cdots s_k)] \circ 
\lambda[\otimes(s_1\cdots s_k),\otimes(s_{k+1}\cdots s_n)](z)\\
&= \rho[\otimes(s_1\cdots s_{j-1}),s_j]\circ 
\lambda[\otimes(s_1\cdots s_j),\otimes(s_{j+1}\cdots s_k)] \circ 
\lambda[\otimes w,\otimes u](z)\\  
&= \rho[\otimes(s_1\cdots s_{j-1}),s_j]\circ 
\lambda[\otimes(s_1\cdots s_j),\otimes(s_{j+1}\cdots s_k)](y)\\
&= v'_j
\end{align*}
and finally
\begin{align*}
v_k &= \rho[\otimes(s_1\cdots s_{k-1}), s_k]
\circ \lambda[\otimes(s_1\cdots s_k),\otimes(s_{k+1}\cdots s_n)](z)\\
&= \rho[\otimes(s_1\cdots s_{k-1}), s_k]
\circ \lambda[\otimes w,\otimes u](z)\\  
&= \rho[\otimes(s_1\cdots s_{k-1}),s_k](y)\\
&= v'_k
\end{align*}
as needed. Commutativity of the right diagram is verified analogously.
\end{proof}

\begin{theorem}\label{divide}
Let $\mathcal{S}$ be a $\lambda\rho$-system over a semigroup\/ $\mathbf{S}$, and
let\/ $\mathbf{H}$ be a semigroup. Let
$(\otimes,\mathbf{t})\colon
(\mathbf{S},\mathcal{S})\to \bigl(S^+,\mathcal{S}^+\bigr)$ 
be the transformation from Definition~\ref{free-t}. Then, 
$$
\mathbf{H}^{(\otimes,\mathbf{t})}\colon \mathbf{H}^{[\mathcal{S}^+]}
\to \mathbf{H}^{[\mathcal{S}]}
$$
of Definition~\ref{hom-from-trans} is a surjective homomorphism. 
\end{theorem}  

\begin{proof}
The map $\mathbf{H}^{(\otimes,\mathbf{t})}$ is a homomorphism by Theorem~\ref{trans}.
Surjectivity follows from Lemma~\ref{well-defd}(1). 
\end{proof}

\subsection{$\lambda\rho$-systems over free monoids}\label{free-constr}
An analogous construction produces a $\lambda\rho$-system
over a free monoid, starting from any system of sets and
maps. Namely, let $\{I[x]:x\in X\}$ be a family of sets, 
let $I$ be a nonempty set, and let $\lambda_x\colon I[x]\to I$ and 
$\rho_x\colon I[x]\to I$ be arbitrary maps. 
Take the free monoid $X^*$, put $I[\varepsilon] = I$, and for each word
$w = x_1x_2\cdots x_k\in X^+$, define $I[{x_1x_2\cdots x_k}]$ to be the set of sequences
$(v_1,v_2,\dots,v_k)\in I[{x_1}] \times \dots \times I[{x_k}]$ such that
\begin{align*}  
  \rho_{x_1}(v_1) &= \lambda_{x_2}(v_2) \\
\rho_{x_2}(v_2) &= \lambda_{x_3}(v_3) \\
 & \vdots  \\
\rho_{x_{k-1}}(v_{k-1}) &= \lambda_{x_k}(v_k). 
\end{align*}
Now define a $\lambda\rho$-system over $X^*$ as follows.
For $w,u\in X^+$ put $\lambda[w,u]$ to be the first projection and
$\rho[w,u]$ to be the second projection, as in Definition~\ref{over-free-sgrp}.
Note, however, that now we only have $I[wu]\subseteq I[w]\times
I[u]$ instead of $I[wu] = I[w]\times I[u]$. The equality holds in particular
cases, for example, if the set $I$ is a singleton.    

It remains to define the maps
$\lambda[\varepsilon, w]$, $\rho[w,\varepsilon]$,
$\lambda[w, \varepsilon]$ and $\rho[\varepsilon, w]$, for any $w\in X^*$.
Put 
$\lambda[w,\varepsilon] = \rho[\varepsilon,w] = id_{I[w]}$
and define the remaining maps inductively. For any $x\in X$ put 
$\lambda[\varepsilon,x] = \lambda_x$, and
$\rho[x,\varepsilon] = \rho_x$.
For $w = \ell r$ with $\ell$ and $r$ nonempty, assuming
$\lambda[\varepsilon,\ell]$ and $\rho[r,\varepsilon]$ have already been defined, 
put $\lambda[\varepsilon,w] =
\lambda[\varepsilon,\ell]\circ\lambda[\ell,r]$
and $\rho[w,\varepsilon] =
\rho[r,\varepsilon]\circ\rho[\ell,r]$.

It can be shown that the resulting system $\mathcal{X}^*$ of sets and maps is
a $\lambda\rho$-system. Figure~\ref{ns-pre-l-r-system} 
illustrates first stages of its construction. If $I$ is a singleton,
then $\mathcal{X}^+$ of Definition~\ref{over-free-sgrp} is a subsystem of
$\mathcal{X}^*$ obtained by deleting $I$ and all the maps into $I$.

\begin{figure}
\begin{tikzpicture}[>=stealth,auto] 
\node (I_xyz) at (0,0) {$I[xyz]$};
\node (I_x) at (-4,-4) {$I[x]$};
\node (I_z) at (4,-4) {$I[z]$};
\node (I_xy) at (-2,-2) {$I[xy]$};
\node (I_yz) at (2,-2) {$I[yz]$};
\node (I_y) at (0,-4) {$I[y]$};
\node (I_1) at (-6,-6) {$I$};
\node (I_2) at (-2,-6) {$I$};
\node (I_3) at (2,-6) {$I$};
\node (I_4) at (6,-6) {$I$};
\draw[dashed,->,bend right=50] (I_xyz)  to node[sloped]
{$\lambda[\varepsilon,xyz]$} (I_1);
\draw[dashed,->,bend left=50] (I_xyz) to node[sloped]
{$\rho[xyz,\varepsilon]$} (I_4);
\draw[dashed,->,bend right] (I_xy) to node[sloped]
{$\lambda[\varepsilon,xy]$} (I_1);
\draw[dashed,->,bend right] (I_yz) to node[pos=0.7,sloped]
{$\lambda[\varepsilon,yz]$} (I_2);
\draw[dashed,->,bend left] (I_xy) to node[pos=0.7,sloped]
{$\rho[xy,\varepsilon]$} (I_3);
\draw[dashed,->,bend left] (I_yz) to node[sloped]
{$\rho[yz,\varepsilon]$} (I_4);
\draw[->] (I_x) to node[swap,sloped] {$\lambda_x = \lambda[\varepsilon,x]$} (I_1);  
\draw[->] (I_x) to node[swap,sloped] {$\rho_x = \rho[x,\varepsilon]$} (I_2);  
\draw[->] (I_y) to node[sloped,swap] {$\lambda_y= \lambda[\varepsilon,y]$} (I_2);  
\draw[->] (I_y) to node[swap,sloped] {$\rho_y = \rho[y,\varepsilon]$} (I_3);  
\draw[->] (I_z) to node[swap,sloped] {$\lambda_z= \lambda[\varepsilon,z]$} (I_3);  
\draw[->] (I_z) to node[swap,sloped] {$\rho_z = \rho[z,\varepsilon]$} (I_4);  
\draw[dashed,->,bend right] (I_xyz) to node[pos=0.3,sloped] {$\lambda[x,yz]$} (I_x);
\draw[dashed,->,bend left] (I_xyz) to node[pos=0.3,sloped] {$\rho[xy,z]$} (I_z);
\draw[dashed,->] (I_xyz) to node[pos=0.7,swap,sloped] {$\lambda[xy,z]$} (I_xy);
\draw[dashed,->] (I_xyz) to node[swap,pos=0.7,sloped] {$\rho[x,yz]$} (I_yz);
\draw[dashed,->] (I_xy) to node[pos=0.7,swap,sloped] {$\lambda[x,y]$} (I_x);
\draw[dashed,->] (I_xy) to node[pos=0.3,swap,sloped] {$\rho[x,y]$} (I_y);
\draw[dashed,->] (I_yz) to node[pos=0.7,swap,sloped] {$\rho[y,z]$} (I_z);
\draw[dashed,->] (I_yz) to node[pos=0.3,swap,sloped] {$\lambda[y,z]$} (I_y);
\end{tikzpicture}
\caption{A system of sets and maps extending to a $\lambda\rho$-system}
\label{ns-pre-l-r-system}
\end{figure}
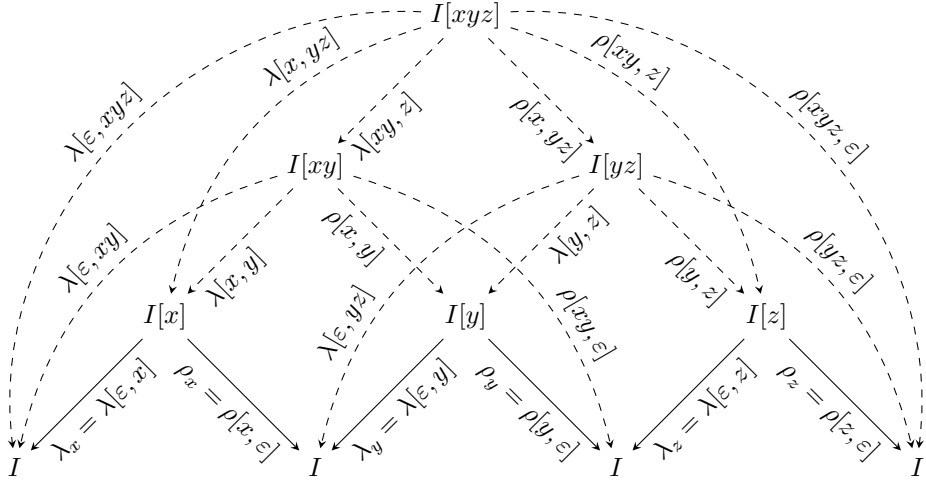

\section{$\lambda\rho$-systems over monoids}\label{simpli}

If $\mathbf{S}$ is a monoid (with identity element $1$),
then any $\lambda\rho$-system
constructed over $\mathbf{S}$ will contain a set $I[1]$, and maps 
$\lambda[a,1]$, $\lambda[1,a]$, $\rho[a,1]$, $\rho[1,a]$ for any $a\in S$.
It is immediate from the conditions
($\alpha$), ($\beta$) and ($\gamma$) that 
the maps $\rho[1,a]$ and $\lambda[a,1]$ are
commuting retractions, that is, they satisfy
\begin{itemize}  
\item $\lambda[a,1]\circ\lambda[a,1] = \lambda[a,1]$
\item $\rho[1,a]\circ\rho[1,a] = \rho[1,a]$
\item $\rho[1,a]\circ\lambda[a,1]  =  \lambda[a,1]\circ\rho[1,a]$
\end{itemize}
for each $a\in S$. In fact, for monoids it is 
reasonable to require more: that $\rho[1,a]$ and $\lambda[a,1]$ are identity maps.
We will now define a general preservation 
requirement, whose special case will apply to monoids. 

\begin{definition}\label{P-preserving}
Let $P$ be a property of semigroups, and let
$\mathcal{S} = (\mathbf{I},\blambda,\brho)$ be a $\lambda\rho$-system
over $\mathbf{S}$. 
We will say that $\mathcal{S}$ \emph{preserves} $P$
\textup{(}or, \emph{is $P$ preserving}\textup{)}, if 
for every $\mathbf{H}$, whenever $\mathbf{H}$ satisfies $P$,
so does $\mathbf{H}^{[\mathcal{S}]}$. 
\end{definition}  

Said concisely, $\mathcal{S}$ is $P$ preserving, if
$\forall\mathbf{H}\colon P(\mathbf{H})\Rightarrow P(\mathbf{H}^{[\mathcal{S}]})$.
If $P$ is the property of having a unit, then 
$\mathcal{S}$ is $P$ preserving (\emph{unit-preserving})
if and only if $\mathbf{H}^{[\mathcal{S}]}$ is a monoid, for every monoid $\mathbf{H}$.

\begin{theorem}\label{main-monoid}
Let $\mathcal{S} = (\mathbf{I},\blambda,\brho)$ be a
$\lambda\rho$-system over $\mathbf{S}$. The following are equivalent:
\begin{enumerate}
\item $\mathcal{S}$ is unit-preserving,
\item $\mathbf{S}$ is a monoid \textup{(}with unit element $1$\textup{)} and the maps
$\lambda[a,1]$ and $\rho[1,a]$ are the identity maps on $I[a]$, for each $a\in S$,
\item $\mathbf{S}$ is a monoid and there exists a nontrivial monoid $\mathbf{H}$
    such that $\mathbf{H}^{[\mathcal{S}]}$  is a monoid.
\end{enumerate}
\end{theorem}

\begin{proof}
To show that (1) implies (3) we only need to prove that $\mathbf{S}$ is a monoid.  
Take $\mathbf{1}^{[\mathcal{S}]}$ for the
trivial monoid $\mathbf{1}$. Then $\mathbf{1}^{[\mathcal{S}]} \cong \mathbf{S}$
and since $\mathcal{S}$ is unit-preserving, $\mathbf{S}$ is a monoid.  

To show that (3) implies (2) let     
$\mathbf{H}$ be a nontrivial monoid with the unit element $e$ such that
$\mathbf{H}^{[\mathcal{S}]}$ is a monoid. Let $1$ be the unit element of
$\mathbf{S}$ and let 
$(y,b)$ be the unit element of $\mathbf{H}^{[\mathcal{S}]}$.
Then $(y,b)\star(x,1) = (x,1)$
for any $x\in H^{I[1]}$, which implies $b\cdot 1 = 1$, so
$b = 1$. Next, taking $(\overline{e},a)$ for any $a\in S$ (where $\overline{e}$
is the constant map from $I[a]$ to $H$ identically equal to $e$), we have
\begin{align*}
  (\overline{e},a) &= (y,1)\star(\overline{e},a)\\
  &= \bigl((y\circ\lambda[1,a])\cdot(\overline{e}\circ\rho[1,a])\
                   ,a\bigr)\\ 
      &= \bigl((y\circ\lambda[1,a])\cdot\overline{e},\ a\bigr)\\
      &= (y\circ\lambda[1,a],\ a).
\end{align*}
This implies that $y\circ\lambda[1,a]$ is also identically $e$.
Therefore, for any $(x,a)$ we obtain
\begin{align*}
(x,a) &= (y,1)\star(x,a)\\
      &= \bigl((y\circ\lambda[1,a])\cdot(x\circ\rho[1,a]),\ a\bigr)\\
      &= \bigl(\overline{e}\cdot(x\circ\rho[1,a]),\ a\bigr)\\
      &= (x\circ\rho[1,a],\ a),
\end{align*}
and therefore $x = x\circ\rho[1,a]$. This holds for an arbitrary $x$,
and as $x$ can be non-constant by nontriviality of $\mathbf{H}$, 
we have $\rho[1,a] = {id}_{I[a]}$. The proof for $\lambda[a,1]$
follows the same lines, but multiplying by identity on the right.
We begin by expanding the right-hand side of
$(\overline{e}, a) = (\overline{e}, a)\star
(y,1)$ to get that $y\circ \rho[a,1]$ is identically equal to $e$.  
Next, we expand the right-hand side of $(x,a) = (x,a)\star(y,1)$
to get $(x,a) = (x\circ\lambda[a,1], a)$ and thus
$x = x\circ\lambda[a,1]$ for an arbitrary $x$, showing that
$\lambda[a,1] = id_{I[a]}$. This ends the proof of (3) $\Rightarrow$ (2). 

To show that (2) implies (1), let $\mathbf{H}$ be any monoid (with identity
element $e$). 
Since $\mathbf{S}$ is a monoid, the set $I[1]$ exists. If $I[1] = \emptyset$,
then $I[s]=\emptyset$ for all $s\in S$ by Proposition~\ref{ideal}, and
then $\mathbf{H}^{[\mathcal{S}]}\cong \mathbf{S}$ (cf. Example~\ref{empty}). 
Assume $I[1]\neq\emptyset$. Since $\mathbf{H}$ is a
monoid, the constant function $\overline{e}$ belongs to $H^{I[1]}$. 
Then, for an arbitrary $(x,a)$ we have 
\begin{align*}
(\overline{e},1)\star (x,a) &= \bigl((\overline{e}\circ\lambda[1,a])\cdot
                              (x\circ\rho[1,a]),\ 1\cdot a\bigr)\\ 
&= (\overline{e}\cdot (x\circ {id}_{I[a]}),\ a\bigr)\\
&= (x, a)
\end{align*}
showing that $(\overline{e},1)\in H^{I[1]}$ is a left unit. A completely
symmetric argument shows that it is a right unit as well. 
\end{proof}

If a $\lambda\rho$-system satisfies the equivalent conditions of
Theorem~\ref{main-monoid}, we will call it \emph{unital}. This piece of
terminology is, strictly speaking, 
redundant, but we find it conceptually useful as a name for an intrinsic
characterisation of being unit-preserving.
The $\lambda\rho$-system of Example~\ref{lzero} is not unital, but
the one of Example~\ref{flip-flop} is, and so is the $\lambda\rho$-system
constructed at the end of Section~\ref{free-constr}.

\section{$\lambda\rho$-systems over groups}\label{wreath}

We have seen in Proposition~\ref{action-wreath-block} that 
every wreath product can be realised as a $\lambda\rho$-product.
Here we will show that for groups the converse is also true.
Let $\mathbf{G}$ be a group acting on a set $X$ on the right, so that we have
$x\ast e = x$ and $(x\ast a)\ast b = x\ast ab$ for any $x\in X$ and
$a,b\in G$. For any such pair $(X,\mathbf G)$ and any group $\mathbf{H}$ recall that 
their wreath product $\mathbf H\wr (X,\mathbf G)$ is
a semidirect product $\mathbf{H}^X\rtimes  \mathbf{G}$ with multiplication
defined by $(u,g)\star(w,h)= (u\cdot (w\circ(\bl \ast g)),\ gh)$.
It is easy to see that any $(X,\mathbf G)$ defines a
$\lambda\rho$-system
$$
\mathcal{S}(X,\mathbf G) =
\bigl(\langle \lambda[g,h],\rho[g,h]\rangle\colon
I[gh]\to I[g]\times I[h]\bigr),
$$
where $I[g]=X$ for any $g\in G$, and 
\begin{enumerate}
\item $\lambda[g,h] = {id}_X$ for any $g,h\in G$,
\item $\rho[g,h] = \bl\ast g$ for all $g,h\in G$, 
\end{enumerate}
as stated immediately after Proposition~\ref{action-wreath-block}. Then 
$\mathbf H^{[\mathcal S(X,\mathbf G)]}\cong \mathbf{H}\wr (X,\mathbf G)$.

\begin{theorem}\label{wr-prod}
Let $\mathcal{S} = (\mathbf{I},\blambda,\brho)$  be a
$\lambda\rho$-system over a semigroup $\mathbf{G}$.
Then, the following are equivalent:
\begin{enumerate}
\item $\mathcal{S}$ is group-preserving,
\item $\mathbf{G}$ is a group and $\mathcal{S}$ is unital,
\item $\mathbf{G}$ is a group and
  $(\mathbf{G},\mathcal{S})\cong (\mathbf{G},\mathcal{S}(X,\mathbf{G}))$
with $\mathbf{G}$ acting on some set $X$. 
\end{enumerate}
\end{theorem} 
\begin{proof}
Recall from Definition~\ref{P-preserving} that group-preserving means
$\mathbf{H}^{[\mathcal{S}]}$ is a group for any group $\mathbf{H}$.  

\smallskip\noindent
(1) $\Rightarrow$ (2). Group-preserving
$\lambda\rho$-systems preserve units, so 
$\mathcal{S}$ is unital.
Consider the trivial group $\mathbf{1}$.
Since $\mathcal{S}$ is a 
$\lambda\rho$-system over $\mathbf{G}$, we have that
$\mathbf{1}^{[\mathcal S]}\cong \mathbf{G}$, and
since $\mathcal{S}$ is group-preserving, 
$\mathbf{G}$ is a group.

\smallskip\noindent
(2) $\Rightarrow$ (3).
Let $\mathcal{S} = (\mathbf{I},\blambda,\brho)$  be a
unital $\lambda\rho$-system over a group $\mathbf{G}$.
Since $\mathcal{S}$ is unital, we have
$$
{id}_{I[g]}=\lambda[g,e]=\lambda [g,hh^{-1}] =
\lambda [g,h]\circ \lambda[gh,h^{-1}]
$$
for all $g,h\in G$ ($e$ is the unit of $\mathbf{G}$, of course).
Consequently, $\lambda [g,h]$ is surjective
and $\lambda [gh,h^{-1}]$ is injective for all $g,h\in G$. However,
$\lambda[g,h]=\lambda [ghh^{-1},h]$ and thus $\lambda[g,h]$ is a
bijection. Analogously we can prove bijectivity of $\rho[g,h]$. 
 
\begin{claim}
Consider the pair $(I[e],\mathbf G)$.
The operation $\cdot\colon I[e]\times
G\to I[e]$ defined by  
$$
i\cdot g = (\rho[g,e]\circ\lambda[e,g]^{-1})(i)
$$
is a group action.
\end{claim}
\begin{proof}[Proof of claim.]
We have $i\cdot e = (\rho[e,e]\circ\lambda[e,e]^{-1})(i) = i$ since
$\rho[e,e]$ and $\lambda[e,e]$ are identity maps, as $\mathcal{S}$ is unital.   
Since $e$ is the unit of $\mathbf{G}$, we have 
$\lambda[eg,h]=\lambda[ge,h]$ and $\rho[g,he]=\rho[g,eh]$, and so we can
substitute the equalities
\begin{eqnarray*}
\lambda[eg,h]&=&\lambda[e,g]^{-1}\circ \lambda[e,gh]\\
\rho[g,he]&=&\rho[h,e]^{-1}\circ \rho[gh,e]
\end{eqnarray*}
into the equality
$$
\lambda[e,h]\circ \rho[g,eh]=\rho[g,e]\circ\lambda[ge,h]
$$
to obtain
$$
\lambda [e,h]\circ \rho[h,e]^{-1}\circ
\rho[gh,e]=\rho[g,e]\circ\lambda[e,g]^{-1}\circ\lambda[e,gh]
$$
and hence
$$
\rho[h,e]\circ\lambda[e,h]^{-1} \circ\rho[g,e]\circ  \lambda
[e,g]^{-1}=\rho[gh,e]\circ \lambda[e,gh]^{-1}.\eqno{(\ddag)}
$$
It is easy to see that the last equality implies $(i\cdot g)\cdot h= i\cdot
(gh)$ for all $i\in I[e]$ and $g,h\in G$. 
\end{proof}

Now we will show that the system of bijections
$\mathbf{t} = \bigl(\lambda[e,g]\colon I[g]\to I[e]\bigr)_{g\in G}$
induces a transformation $(id_G, \mathbf{t})\colon
(\mathbf{G},\mathcal{S}) \to \mathcal{S}(X,\mathbf{G})$
with $id_G$ the identity map on $G$; hence the desired isomorphism. 
We need to prove commutativity of the following diagrams: 
$$
\begin{tikzpicture}[>=stealth,auto]
\node (tab) at (0,0) {$I[gh]$};
\node (ab) at (3,0) {$I[e]$};
\node (ta) at (0,-2) {$I[g]$};
\node (a) at (3,-2) {$I[e]$};
\draw[->] (tab) to node {$\lambda[e,gh]$} (ab);
\draw[->] (tab) to node[swap] {$\lambda [g,h]$} (ta);
\draw[->] (ab) to node[swap] {$id_{I[e]}$} (a);
\draw[->] (ta) to node {$\lambda[e,g]$} (a);
\end{tikzpicture} 
\qquad
\begin{tikzpicture}[>=stealth,auto]
\node (tab) at (0,0) {$I[gh]$};
\node (ab) at (4,0) {$I[e]$};
\node (tb) at (0,-2) {$I[h]$};
\node (b) at (4,-2) {$I[e].$};
\draw[->] (tab) to node {$\lambda[e,gh]$} (ab);
\draw[->] (tab) to node[swap] {$\rho[g,h]$} (tb);
\draw[->] (ab) to node[swap] {$\rho[g,e]\circ\lambda[e,g]^{-1}$} (b);
\draw[->] (tb) to node {$\lambda[e,h]$} (b);
\end{tikzpicture} 
$$
Commutativity of the first diagram is clear. Composing
both sides of $(\ddag)$ on the left with $\rho[h,e]^{-1}$
and using $\rho[g,h]=\rho[h,e]^{-1}\circ\rho[gh,e]$, we obtain 
$$\lambda[e,h]^{-1}\circ\rho[g,e]\circ\lambda[e,g]^{-1}=\rho[g,h]\circ\lambda [e,gh]^{-1}$$
which proves commutativity of the second diagram. 

\smallskip\noindent
(3) $\Rightarrow$ (1). Follows from the fact that wreath product of groups
is a group. 
\end{proof}

Theorem~\ref{wr-prod} shows that $\lambda\rho$-products over groups coincide
with wreath products, as long as they always produce groups.
We saw that for semigroups the notion of a $\lambda\rho$-products is more
general. Combining Krohn-Rhodes Theorem, Theorem~\ref{wr-prod}, and
Example~\ref{flip-flop}, we get a little application.

\begin{cor}
Every finite semigroup divides an iterated
$\lambda\rho$-product whose factors are finite simple groups and a two-element
semilattice.
\end{cor}

\section{Acknowledgements}

The authors sincerily thank the anonymous referee for several rounds of detailed
comments. They were invaluable for improving both the presentation and the
content. 

This research has received funding from the European Union’s \emph{Horizon 2020}
research and innovation programme under the Marie Skłodowska-Curie grant
agreement No.~689176. The first author was also supported by
the bilateral project \emph{New perspectives on residuated posets} of the
Austrian Science Fund (FWF): project I 1923-N25, and the Czech Science
Foundation (GACR): project 15-34697L.

\end{document}